\documentclass[12pt]{amsart}
\usepackage{amsmath,amssymb,amsbsy,amsfonts,amsthm,latexsym,amsopn,amstext,
                                                    amsxtra,euscript,amscd}
\usepackage[english]{babel}
\begin{document}

\newtheorem{thm}{Theorem}
\newtheorem{lem}[thm]{Lemma}
\newtheorem{claim}[thm]{Claim}
\newtheorem{cor}[thm]{Corollary}
\newtheorem{prop}[thm]{Proposition} 
\newtheorem{definition}{Definition}
\newtheorem{question}[thm]{Open Question}
\newtheorem{conj}[thm]{Conjecture}
\newtheorem{prob}{Problem}
\def\vol {{\mathrm{vol\,}}}
\def\squareforqed{\hbox{\rlap{$\sqcap$}$\sqcup$}}
\def\qed{\ifmmode\squareforqed\else{\unskip\nobreak\hfil
\penalty50\hskip1em\null\nobreak\hfil\squareforqed
\parfillskip=0pt\finalhyphendemerits=0\endgraf}\fi}

\def\cA{{\mathcal A}}
\def\cB{{\mathcal B}}
\def\cC{{\mathcal C}}
\def\cD{{\mathcal D}}
\def\cE{{\mathcal E}}
\def\cF{{\mathcal F}}
\def\cG{{\mathcal G}}
\def\cH{{\mathcal H}}
\def\cI{{\mathcal I}}
\def\cJ{{\mathcal J}}
\def\cK{{\mathcal K}}
\def\cL{{\mathcal L}}
\def\cM{{\mathcal M}}
\def\cN{{\mathcal N}}
\def\cO{{\mathcal O}}
\def\cP{{\mathcal P}}
\def\cQ{{\mathcal Q}}
\def\cR{{\mathcal R}}
\def\cS{{\mathcal S}}
\def\cT{{\mathcal T}}
\def\cU{{\mathcal U}}
\def\cV{{\mathcal V}}
\def\cW{{\mathcal W}}
\def\cX{{\mathcal X}}
\def\cY{{\mathcal Y}}
\def\cZ{{\mathcal Z}}

\def\NmQR{N(m;Q,R)}
\def\VmQR{\cV(m;Q,R)}

\def\Xm{\cX_m}

\def \C {{\mathbb C}}
\def \F {{\mathbb F}}
\def \L {{\mathbb L}}
\def \K {{\mathbb K}}
\def \Q {{\mathbb Q}}
\def \R {{\mathbb R}}
\def \Z {{\mathbb Z}}
\def \fS{\mathfrak S}

\def\\{\cr}
\def\({\left(}
\def\){\right)}
\def\fl#1{\left\lfloor#1\right\rfloor}
\def\rf#1{\left\lceil#1\right\rceil}

\def\Tr{{\mathrm{Tr}}}
\def\Im{{\mathrm{Im}}}

\def \bFp {\overline \F_p}

\newcommand{\pfrac}[2]{{\left(\frac{#1}{#2}\right)}}

\def \Prob{{\mathrm {}}}
\def\e{\mathbf{e}}
\def\ep{{\mathbf{\,e}}_p}
\def\epp{{\mathbf{\,e}}_{p^2}}
\def\em{{\mathbf{\,e}}_m}

\def\Res{\mathrm{Res}}

\def\vec#1{\mathbf{#1}}
\def\flp#1{{\left\langle#1\right\rangle}_p}

\def\mand{\qquad\mbox{and}\qquad}

\newcommand{\comm}[1]{\marginpar{%
\vskip-\baselineskip 
\raggedright\footnotesize
\itshape\hrule\smallskip#1\par\smallskip\hrule}}

\title{Polynomial Values in Small Subgroups of Finite Fields}

\author{Igor E. Shparlinski} 
\address{Department of Pure Mathematics, University of New South Wales, 
Sydney, NSW 2052, Australia}
\email{igor.shparlinski@unsw.edu.au}

\date{\today}

\begin{abstract} For a large prime $p$, 
and a polynomial $f$ over a finite field 
$\F_p$ of $p$ elements,  we obtain 
a lower bound on the size of the multiplicative 
subgroup of $\F_p^*$ containing $H\ge 1$ consecutive values $f(x)$, 
$x = u+1, \ldots, u+H$, uniformly over $f\in\F_p[X]$ and an $u \in \F_p$. 
\end{abstract}



\maketitle

\section{Introduction}

\subsection{Background}

For a prime $p$, we use $\F_p$ to denote the finite field  of $p$ elements,
which we always assume to be represented by the set $\{0, \ldots, p-1\}$.

For a rational function $r(X) = f(X)/g(X)\in\F_p(X)$ with two relatively 
primes polynomials $f,g \in \F_p[X]$
and an  set $\cS \subseteq \F_p$, we use $r(\cS)$ to 
denote   the value set 
$$
r(\cS) = \{r(x)~:~x \in \cS, \ g(x) \ne 0\} \subseteq \F_p.
$$
Given for two sets $\cS, \cT  \subseteq \F_p$, 
we consider the size of the intersection of $r(\cS)$ and $\cT$, 
that is, 
$$
N_r(\cS,\cT) =\#\(r(\cS) \cap \cT\).
$$

Here, we are mostly interested in studying 
$N_r(\cI,\cG)$ for an interval $\cI$ 
of several  consecutive integers and 
a multiplicative subgroup $\cG$ of $\F_p^*$. 

We also use $T_r(H)$ to denote the
smallest possible $T$ such that there is an 
interval  $\cI = \{u+1, \ldots, u+H\}$ of $H$ consecutive integers 
and a multiplicative subgroup $\cG$ of $\F_p^*$
of order $T$ for which 
$$
r(\cI) \subseteq \cG
$$
and thus $N_r(\cI, \cG) = \# r(\cI)$.

It is shown in~\cite{GomShp} that  if  $f\in \F_p[X]$ 
is a polynomial of degree $d\ge 2$, 
then for any interval $\cI = \{u+1, \ldots, u+H\}$  of   
$H$ consecutive
integers and a subgroup $\cG$ of $\F_p^*$ of order $T$, 
the quantity $N_f(\cI, \cG)$ is ``small''.

 To formulate 
the result precisely we recall that
the notations $U = O(V)$, $U \ll V$ and  $V \gg U$  are all
equivalent to the inequality $|U| \le c\,V$ 
with some constant $c> 0$. 
Throughout the paper, the implied constants
in these symbols
may occasionally depend, where obvious, on  degrees and the number 
of variables of various polynomials, but are absolute otherwise.
We also use $o(1)$ to denote a quantity that tends to zero when 
one of the indicated parameters (usually $H$ or $p$) tends to infinity. 

Then, by the bound of~\cite{GomShp}, for  
 a polynomial  $f\in \F_p[X]$  of degree $d\ge 2$, 
we have 
\begin{equation}
\label{eq:NfIG}
N_f(\cI, \cG)\ll (1 + H^{(d+1)/4} p^{-1/4d}) H^{1/2d} T^{1/2}. 
\end{equation}
In particular, the bound~\eqref{eq:NfIG} implies that 
\begin{equation}
\label{eq:TfH}
T_f(H) \gg \min\{H^{2 - 1/d}, H^{-(d-1)(d-2)/2d}p^{1/2d} \}.
\end{equation}
For a linear fractional function 
$$
r(X) = a \frac{X+s}{X+t}
$$
with $s \not \equiv t \pmod p$, 
the bound of~\cite[Lemma~35]{BGKS1} implies that there is an 
absolute constant $c>0$ such that if for some positive integer $\nu$
we have 
$$
H \ge p^{c\nu^{-4}}
$$
then 
\begin{equation}
\label{eq:TrH-Mobius}
T_r(H) \gg H^{\nu + o(1)},
\end{equation}
as $H\to \infty$. 
For larger values of $H$, by~\cite[Bound~(29)]{BGKS1} we have
$$
N_r(\cI, \cG)\le \(1+  H^{3/4}p^{-1/4}\)T^{1/2} p^{o(1)},
$$
as $p\to \infty$. Thus
$$
T_r(H) \ge   \min\{H^{2} ,H^{1/2}p^{1/2} \} p^{o(1)},
$$

A series of other upper bounds on $N_{r}(\cS, \cT)$ 
and its multivariate generalisations, for various sets and $\cS$ and $\cT$
(such as intervals, subgroups, zero-sets of algebraic varieties and 
their Cartesian products)
and functions $r$, including multivariate functions,  are given 
in~\cite{Bour,BGKS1,BGKS2, Chang1,Chang2,Chang3,CCGHSZ,CKSZ,CGOS,Shp-GG}.  

\subsection{Our results} 

Here we use the methods of~\cite{BGKS1}, based on an application
effective Hilbert's Nullstellensatz, see~\cite{DKS,KiPaSo}, 
to obtain a variant of the bound of~\eqref{eq:TrH-Mobius} for 
polynomials and thus 
to improve~\eqref{eq:TfH} for small values of $H$.

Furthermore, combining some ideas from~\cite{GomShp} with a bound 
on the number on integer points on quadrics (which replaces the 
bound of Bombieri and Pila~\cite{BP} in the argument of~\cite{GomShp}), 
we  improve~\eqref{eq:NfIG}  for quadratic polynomials. In fact, this argument   stems from 
that of  Cilleruelo and Garaev~\cite{CillGar}.

\section{Preparations}

\subsection{Effective Hilbert's Nullstellensatz}

We recall that the logarithmic height of a nonzero polynomial $P \in
\Z[Z_1, \ldots, Z_n]$ is defined as the maximum
logarithm of the largest (by absolute value) coefficient of $P$.

Our argument uses  the following quantitative version
version of effective Hilbert's Nullstellensatz
due to   Krick,  Pardo and   Sombra~\cite[Theorem~1]{KiPaSo}.

\begin{lem}
\label{lem:Hilb} Let
$P_1, \ldots, P_N \in \Z[Z_1, \ldots, Z_n]$
be $N\ge 1$ polynomials in $n$ variables
without a common zero in $\C^n$
of degree at most $D\ge 3$ and of logarithmic height
at most $H$. Then there is a positive integer $b$
with
$$
\log b \le 4 n(n+1) D^{n}\(H + \log N + n D \log(n + 1)\)
$$
and polynomials $R_1, \ldots, R_N\in \Z[Z_1, \ldots, Z_n]$ such that
$$
P_1R_1+ \ldots + P_NR_N = b.
$$
\end{lem}

We note that~\cite[Theorem~1]{KiPaSo} gives explicit estimates 
on all other parameters as well (that is, on the heights and 
degrees of the  polynomials $R_1, \ldots, R_N$), see also~\cite{DKS}. 

\subsection{Some facts on algebraic integers}

We also need a bound of Chang~\cite[Proposition~2.5]{Chang0} on the divisor function in
algebraic number fields.
As usual, for algebraic number field $\K$ we use $\Z_\K$ 
to denote the ring of integers. 

\begin{lem}
\label{lem:Div ANF} Let $\K$ be a finite extension of $\Q$ of degree
$k = [\K:\Q]$. For
any nonzero algebraic integer $\gamma\in \Z_\K$ of logarithmic
height at most $H\ge 2$, the number of   pairs $(\gamma_1, \gamma_2)$
of  algebraic integers $\gamma_1,\gamma_2\in \Z_\K$ of
logarithmic  height at most $H$
with $\gamma=\gamma_1\gamma_2$ is at most $\exp\(O(H/\log H)\)$,
where the implied constant depends on $k$.
\end{lem}

Finally, as in~\cite{BGKS1}, we use the following 
result, this is exactly the statement that
is established in the proof of~\cite[Lemma~2.14]{Chang0},
see~\cite[Equation~(2.15)]{Chang0}.

\begin{lem}
\label{lem:SmallZero} Let $P_1, \ldots, P_N, P \in \Z[Z_1, \ldots,
Z_n]$ be $N+1 \ge 2$ polynomials in $n$ variables of degree at
most $D$ and of logarithmic height at most $H\ge1$. If  the
zero-set
$$
P_1(Z_1, \ldots, Z_n) = \ldots =P_N(Z_1, \ldots, Z_n) = 0 \quad
\text{and}\quad
 P(Z_1, \ldots, Z_n) \ne  0
$$
is not empty then it has a point $(\beta_1, \ldots, \beta_n)$
in an extension $\K$ of $\Q$ of degree $[\K:\Q]\le C_1(D,n)$
such that their  minimal polynomials are  of  logarithmic
height at most $C_2(D,N,n) H$,
where $C_1(D,n)$ depends only on $D$, $n$ and $C_2(D,N,n)$ depends
only on $D$, $N$ and $n$.
\end{lem}

\subsection{Integral points on quadrics}

The following bound on the number of integral points on 
quadrics is given in~\cite[Lemma~3]{KonShp}. 

\begin{lem}
\label{lem:Quadratic}
Let
$$
G(X,Y)= AX^2  + BXY + C Y^2 + DX + EY + F\in \Z[X,Y]
$$
be an irreducible quadratic polynomial with coefficients of
size at most $H$. Assume that $G(X,Y)$ is not 
affinely 
equivalent to a parabola $Y = X^2$ and has a nonzero determinant
$$
\Delta = B^2-4AC \ne 0.
$$
 Then, as $H\to \infty$, the equation $G(x,y)=0$ has
at most $H^{o(1)}$ integral solutions $(x,y) \in [0,H]\times[0,H]$. 
\end{lem}

\subsection{Small values of linear functions}

We need a result about small values of residues modulo $p$
of several linear functions. Such a result has been 
derived in~\cite[Lemma~3.2]{CSZ} from 
 the Dirichlet pigeon-hole principle.
Here use a slightly more precise and explicit 
form of this result which is derived in~\cite{GG} 
from the {\it Minkowski theorem\/}, see also~\cite{GomShp}.

For an integer $a$ we use $\flp{a}$ to denote the smallest by
absolute value residue of $a$ modulo $p$, that is
$$
\flp{a} = \min_{k\in \Z} |a - kp|.
$$

\begin{lem}
\label{lem:Red}
For any real numbers $V_1,\ldots, V_{m} $ with
$$
p> V_1,\ldots, V_{m} \ge 1 \mand   V_1\ldots V_{m} > p^{m-1}
$$
and  
integers $b_1, \ldots, b_{m}$,  there exists an integer $v$ with $\gcd(v,p) =1$
such that
$$
\flp{b_i v} \le V_i, \qquad i =1, \ldots, m.
$$
\end{lem}

\section{Main Results}

\subsection{Arbitrary polynomials} 

For a set $\cA$ in an arbitrary semi-group, we use  $\cA^{(\nu)}$ to denote 
the $\nu$-fold product set, that 
is
$$
\cA^{(\nu)}=\{a_1\ldots a_\nu~:~ a_1,\ldots, a_\nu \in \cA\}.
$$

First we note that in order to get a lower bound on  $T_f(\cI, \cG)$ 
it is enough to give a lower bound on the cardinality of  $f(\cI)^{(\nu)}$
for any integer $\nu \ge 1$.

\begin{thm} \label{thm:N poly} For every positive integer $\nu$ 
there is a constant  $c(\nu) > 0$ 
depending only on $\nu$ such that for any  
polynomial  $f\in \F_p[X]$ of degree $d\ge 1$ 
and interval $\cI$  of   
$$
H \le c(\nu) p^{1/4d(d+1)\nu^{d+1}}
$$ 
consecutive integers,  we have 
$$
\# f(\cI)^{(\nu)}\ge H^{\nu + o(1)},
$$
as $H\to \infty$.
\end{thm}

\begin{proof}
Clearly, we can assume that 
$$f(X)= X^d + \sum_{k=0}^{d-1} a_{d-k} X^k
$$ 
is monic. 

It is also clear that we can assume that $\cI = \{1, \ldots, H\}$. 

We consider the collection $\cP\subseteq \Z[Z_1,  \ldots, Z_d]$
of  polynomials
$$
P_{\vec{x},\vec{y}}(Z_1,  \ldots, Z_d) =
\prod_{i=1}^\nu  \(x_i^d + \sum_{k=0}^{d-1} Z_{d-k} x_i^k\)
-  \prod_{i=1}^\nu \(y_i^d + \sum_{k=0}^{d-1} Z_{d-k} y_i^k\),
$$
where $\vec{x} = (x_1, \ldots, x_\nu)$ and
$\vec{y} = (y_1, \ldots, y_\nu)$ are integral vectors with
entries in $[1,H]$ and such that
$$
P_{\vec{x},\vec{y}}(a_1, \ldots, a_k) \equiv 0 \pmod p.
$$
Note that 
$$
P_{\vec{x},\vec{y}}(a_1, \ldots, a_k) \equiv \prod_{i=1}^\nu  f(x_i)
-\prod_{i=1}^\nu  f(y_i)\pmod p.
$$

Clearly if $P_{\vec{x},\vec{y}}$ is identical to zero then, by the 
uniqueness of polynomial factorisation in the ring $ \Z[Z_1,  \ldots, Z_d]$, 
the components of $\vec{y}$ are permutations of those of $\vec{x}$. 
So in this case we obviously obtain 
$$
\# f(\cI)^{(\nu)}\ge \frac{1}{\nu!} \(\# f(\cI)\)^{\nu} \gg H^\nu.
$$

Hence, we now assume that
 $\cP$ contains non-zero polynomials.

Note that every $P\in \cP$ is of degree at most $\nu$
and of logarithmic height at most $\nu \log H + O(1)$.

We take a family $\cP_0$ containing the largest possible number
$$
N \le (\nu +1)^d
$$
of linearly independent polynomials $P_1, \ldots, P_N \in \cP$,  and consider the
variety
$$
\cV: \ \{(Z_1,  \ldots, Z_d) \in \C^d~:~P_1(Z_1,  \ldots, Z_d) = \ldots =P_N(Z_1,  \ldots, Z_d) = 0\}.
$$

Assume that $\cV = \emptyset$.  Then by
Lemma~\ref{lem:Hilb} we see that there are polynomials
$R_1, \ldots, R_N \in \Z[Z_1,  \ldots, Z_d]$ and a positive integer
$b$ with
\begin{equation}
\label{eq:b small birat}
\log b \le  4d(d+1) \nu^{d}(\nu \log h  + O(1)) \le 4d(d+1) \nu^{d+1} \log h + O(1)
\end{equation}
and such that
\begin{equation}
\label{eq:b lin comb}
P_1R_1+ \ldots + P_NR_N = b 
\end{equation}

Substituting 
$$(Z_1,  \ldots, Z_d) = (a_1, \ldots, a_k)
$$ 
in~\eqref{eq:b lin comb}, we see that the left hand side of~\eqref{eq:b lin comb}
is divisible by $p$. Since $b \ge 1$ we obtain $p\le b$.  Taking an appropriately small
values of $c(\nu)$ in the condition of the theorem, we see
from~\eqref{eq:b small birat} that this is impossible.

Therefore the variety $\cV$ 
is nonempty. Applying Lemma~\ref{lem:SmallZero}
we see that  it has a point $(\beta_1, \ldots, \beta_d)$
with components of   logarithmic height $O(\log h)$
in an extension $\K$ of $\Q$ of degree $[\K:\Q] = O(1)$.

 Consider the maps
$\Phi:\  \cI^\nu  \to \F_p$
given by
$$
\Phi: \ \vec{x} = (x_1, \ldots, x_\nu) \mapsto \prod_{j=1}^\nu f(x_j)
$$
and $\Psi:  \cI^\nu  \to \K$
given by
$$
\Psi: \ \vec{x} = (x_1, \ldots, x_\nu) \mapsto  \prod_{j=1}^\nu \(x_i^d + \sum_{k=0}^{d-1} \beta_{d-k} x_i^k\).
$$
By construction of  $(\beta_1, \ldots, \beta_d)$ we have that
if $\Phi(\vec{x}) = \Phi(\vec{y})$ then 
$$P_{\vec{x},\vec{y}}(a_1, \ldots, a_k)\equiv 0 \pmod p,
$$
thus $P_{\vec{x},\vec{y}}(Z_1,  \ldots, Z_d)  \in \cP$. 
Recalling the definitions of the family $\cP_0$ and of $(\beta_1, \ldots, \beta_d)$, 
we see that 
$P_{\vec{x},\vec{y}}(\beta_1, \ldots, \beta_d) = 0$.
Hence $\Psi(\vec{x}) = \Psi(\vec{y})$. We now conclude that 
for every $\vec{x}$ the multiplicity 
of the value $\Phi(\vec{x})$ in the 
image set $\Im \Phi$  of the map $\Phi$
is at most the multiplicity 
of the value $\Phi(\vec{x})$ in the 
image set $\Im \Psi$  of the map $\Psi$. 
Thus, 
$$
\# f(\cI)^{(\nu)}= \# \Im \Phi\ge \# \Im \Psi = \# \cC^{(\nu)},
$$
where 
$$
\cC = \left\{x^d + \sum_{k=0}^{d-1} \beta_{d-k}x^d~:~ 1\le x\le H \right\}
\subseteq \K.
$$
Using Lemma~\ref{lem:Div ANF}, we conclude that  
$\# \cC^{(\nu)} \ge H^{\nu + o(1)}$, as $H\to \infty$, and derive the result.
\end{proof}

\subsection{Quadratic polynomials} 

For quadratic square-free polynomials $f$ using 
Lemma~\ref{lem:Quadratic} instead of the 
bound of Bombieri and Pila~\cite{BP} in the argument 
of~\cite{GomShp} we immediately obtain the following result. 

\begin{thm} \label{thm:N quadr}
Let   $f(X)\in \F_p[X]$ be a square-free quadratic polynomial. 
For any interval $\cI$  of   
$H$ consecutive
integers and a subgroup $\cG$ of $\F_p^*$ of order $T$, 
we have 
$$
N_f(\cI, \cG)\le \(1+ H^{3/4}p^{-1/8}\)T^{1/2} p^{o(1)},
$$
as $H\to \infty$.
\end{thm}

\begin{proof} We follow closely  the argument of~\cite{GomShp}. 
We can assume  that 
\begin{equation}
\label{eq:const c}
H \le c p^{1/2}
\end{equation}
for some constant $c>0$ 
as otherwise the desired bound 
is weaker than the trivial estimate 
$$
N_{f}(\cI, \cG)\le \min\{H,T\} \le H^{1/2} T^{1/2}.
$$

Making the transformation $X \mapsto X+u$
we reduce the problem to the case where $\cI = \{1, \ldots, H\}$.

Let $1 \le x_1 < \ldots < x_k \le H$ be all $k=N_f(\cI, \cG)$ values of $x \in \cI$
with  $f(x) \in \cG$. 

Let $f(X)= a_0X^d + a_1X + a_2$, $a_0 \ne 0$.

Let us consider the quadric 
\begin{equation}
\label{eq:quadr Q}
\begin{split}
Q_\lambda(X,Y)& = f(X)- \lambda f(Y)=\\
& = 
a_0X^2 - \lambda a_0 Y^2 + a_1X- \lambda a_1 Y + a_2(1-\lambda). 
\end{split}
\end{equation}
One easily verifies that 
$Q_\lambda(X,Y) $ is irreducible for $\lambda\ne 1$.

We see that there are only at most $2k$ pairs $(x_i,x_j)$, $1 \le
i,j \le r$, 
for which  $f(x_i)/f(x_j) =1$. Indeed, if $x_j$ is fixed, then
$f(x_i)$ is also fixed, and thus $x_i$ can take  
at most $2$ values.

We now assume that $k \ge  4$ as otherwise there is nothing to prove. 
Therefore, 
there is $\lambda \in \cG \setminus \{1\}$  such that 
\begin{equation}
\label{eq:cong f}
f(x) \equiv \lambda f(y) \pmod p
\end{equation}
for at least 
\begin{equation}
\label{eq:pairs}
\frac{k^2 -2k}{T} \ge \frac{k^2}{2T} 
\end{equation}
pairs $(x,y)$ with $x,y \in \{1, \ldots, H\}$. 

We now apply Lemma~\ref{lem:Red} with $s = 4$,  
$$
b_1 = a_0 \quad b_2 = -\lambda a_0 ,\quad  b_3 = a_1, \quad b_4 = -\lambda a_1
$$
and
$$
V_1 = V_2 = 2p^{3/4} H^{-1/2}, \qquad V_3 = V_4 = 2p^{3/4} H^{1/2}.
$$
Thus
$$
V_1 V_2 V_3 V_4= 16 p^{3} > p^{3}.
$$
We also assume that the constant $c$  in~\eqref{eq:const c}
is small enough so the condition 
$$
V_i \le 2p^{3/4} H^{1/2} <p, \qquad i =1, \ldots, 4,
$$
is satisfied. Note that 
\begin{equation}
\label{eq:ViQ}
V_1H^2 =V_2H^2 =  V_3H = V_4H = 2p^{3/4} H^{3/2}. 
\end{equation}

Let $F(X,Y) \in \Z[X,Y]$ be the quadric 
with coefficients in the interval $[-p/2,p/2]$, 
obtained by reducing $vQ_\lambda(X,Y)$ modulo $p$.
Clearly~\eqref{eq:cong f} implies 
\begin{equation}
\label{eq:Cong FG}
F(x,y) \equiv  0 \pmod p.
\end{equation}
Furthermore,  since for $x,y\in \{1, \ldots, H\}$,
recalling~\eqref{eq:cong f}, 
we see from~\eqref{eq:ViQ} and the trivial estimate 
on the constant coefficient (that is, $|F(0,0)| \le p/2$)
that 
$$
|F(x,y)| \le 8 p^{3/4} H^{3/2} + p/2, 
$$ 
which together with~\eqref{eq:Cong FG} implies that 
\begin{equation}
\label{eq:Eq FGz}
F(x,y) - zp = 0
\end{equation}
for some integer $z \ll 1+  H^{3/2}p^{-1/4}$.

Clearly, for any integer $z$ the reducibility of  $F(X,Y)-pz$  over $\C$ implies the
reducibility of $F(X,Y)$ and then in turn of $Q_\lambda(X,Y)$ over $\F_p$, which is 
impossible  as $\lambda \ne 1$.  Hence, Lemma~\ref{lem:Quadratic} implies  
that, as $p\to \infty$, for every $z$ the equation~\eqref{eq:Eq FGz}
has  $p^{o(1)}$ solutions.
Thus the congruence~\eqref{eq:cong f} has at most
$\(1+  H^{3/2}p^{-1/4}\)p^{o(1)}$ solutions.
Together with~\eqref{eq:pairs}, this yields the inequality 
$$
\frac{k^2}{2T} \le\(1+  H^{3/2}p^{-1/4}\)p^{o(1)},
$$ 
and concludes the proof.
\end{proof}


\section*{Acknowledgements}

The author is very grateful to Domingo G\'omez-P\'erez 
for discussions and very helpful comments. 

This work was finished during a very enjoyable  research stay of the 
author at the Max Planck Institute for Mathematics, Bonn.
 
During the preparation of this paper the author was supported by the   
Australian Research Council
Grant~DP130100237.

\end{document}